\documentclass{article}
\usepackage{amsthm,amsmath,amssymb,xcolor,cleveref} 
\newtheorem{theorem}{Theorem}[section]
\newtheorem{conjecture}[theorem]{Conjecture}

\newtheorem{corollary}[theorem]{Corollary}
\newtheorem{lemma}[theorem]{Lemma}
\theoremstyle{definition}

\newcommand{\eq}[1]{\begin{align*}#1\end{align*}}
\newcommand{\eqn}[1]{\begin{align}#1\end{align}}
\newcommand{\eqon}[1]{\begin{align}\begin{split}#1\end{split}\end{align}}

\newcommand{\x}{\times}
\renewcommand{\-}{\setminus}

\newcommand{\R}{\mathbb R}

\newcommand{\ceil}[1]{\left\lceil#1\right\rceil}
\DeclareMathOperator{\dist}{dist^+}
\newcommand{\distprime}{\mathrm{dist}'^+}

\title{An improved bound on Seymour's second neighborhood conjecture}
\author{Hao Huang \thanks{Department of Mathematics, National University of Singapore. Email: huanghao@nus.edu.sg. Research supported in part by a start-up grant at NUS and an MOE Academic Research Fund (AcRF) Tier 1 grant.} \and Fei Peng \thanks{Department of Mathematics, National University of Singapore. Email: pfpf@u.nus.edu.}}
\date{}

\begin{document}

\maketitle
\begin{abstract}
Seymour's celebrated second neighborhood conjecture, now more than thirty years old, states that in every oriented digraph, there is a vertex $u$ such that the size of its second out-neighborhood $N^{++}(u)$ is at least as large as that of its first out-neighborhood $N^+(u)$. In this paper, we prove the existence of $u$ for which $|N^{++}(u)| \ge 0.715538 |N^+(u)|$. This result provides the first improvement to the best known constant factor in over two decades.
\end{abstract}
\section{Introduction}
An \textit{oriented digraph} is a directed graph without digons or loops; that is, there are no vertices $u,v$ for which the arcs $(u,v)$ and $(v,u)$ are both included. A \textit{tournament} is an edge-maximal oriented digraph; that is, between every two distinct vertices $u$ and $v$, exactly one of $(u,v)$ and $(v,u)$ is in the digraph. Given two vertices $u$ and $v$, the \textit{positive distance} from $u$ to $v$, denoted by $\dist(u,v)$, is the length of the shortest non-trivial directed walk starting from $u$ and ending at $v$. We set $\dist(u,v)=\infty$ if no such walk exists. Note that in an oriented digraph, $\dist(u,u)\ge3$. 
    
    Given a vertex $u$ in an oriented digraph $D$, define
    \eq{N^+(u)&=\{v\in V(D):\dist(u,v)=1\},\quad d^+(u)=|N^{+}(u)|;\\
    N^{++}(u)&=\{v\in V(D):\dist(u,v)=2\},\quad d^{++}(u)=|N^{++}(u)|;\\
    N^{+++}(u)&=\{v\in V(D):\dist(u,v)=3\},\quad d^{+++}(u)=|N^{+++}(u)|.}

    As $\dist(u,u)\ge3$, the definitions of $N^+(u)$, $N^{++}(u)$, $d^+(u)$ and $d^{++}(u)$ remain unchanged regardless of whether we use the usual distance or the positive distance. However, it is now possible that $u\in N^{+++}(u)$. We denote $N^+(u)$, $N^{++}(u)$ and $N^{+++}(u)$ as the \textit{first}, \textit{second}, and \textit{third neighborhood} of $u$, respectively. We also denote the vertices in $N^+(u)$ the \textit{out-neighbors} of $u$, and $d^+(u)$ as the \textit{out-degree} of $u$. For a fixed constant $\mu\in\R_{\ge0}$, we say that $u\in V(D)$ is a $\mu$-\textit{Seymour vertex} if $d^{++}(u)\ge \mu d^+(u)$.
    
In the early 1990s, Seymour posed the now-celebrated second neighborhood conjecture: in every digraph, there exists a vertex $u$ such that the size of its second out-neighborhood $N^{++}(u)$ is at least as large as that of its first out-neighborhood $N^+(u)$. Using the above notations, 
the second neighborhood conjecture can be restated as follows:
\begin{conjecture}[see~\cite{dean1995squaring}]\label{main_conj}
    Every oriented digraph has a $1$-Seymour vertex.
\end{conjecture}

 A notable special case of Conjecture \ref{main_conj} asks if it holds for tournaments. This was known as \textit{Dean's conjecture}, and was first verified by Fisher~\cite{fisher1996squaring} based on Farkas' Lemma. Later, Havet and Thomassé~\cite{havet2000median} gave a very different proof using median orders. The second neighborhood conjecture has attracted a lot of attention over the years \cite{ai2024seymour, brantner2009contributions, cohn2016number, dara2022extending, espuny2024seymour, fidler2007remarks, ghazal2012seymour, hernandez2012k, kaneko2001minimum, liang2017seymour, lim2020generalisation, llado2013second, mezher2024note, subset, sullivan2006summary}. Despite various results on special cases and attempts to tackle the original form, Conjecture \ref{main_conj} remains open for general oriented digraphs. If the conjecture is true, the constant $1$ is best possible, as demonstrated by the sparse powers, disjoint unions, and iterative blow-ups of directed cycles. One viable target is to find in a general oriented digraph a $\mu$-Seymour vertex, for some $\mu <1$. It is not hard to show that a vertex with minimum out-degree is always a $\frac12$-Seymour vertex.  With some careful analysis, Chen, Shen and Yuster~\cite{csy} proved that the constant factor $\frac{1}{2}$ can be improved.
\begin{theorem}[\cite{csy}, Theorem~6]\label{thm_csy}
    Every oriented digraph has a $\lambda$-Seymour vertex, where $\lambda=0.657298\dots$ is the unique real root of $2x^3+x^2-1=0$.
\end{theorem}
At the end of their paper, they also claim that a better constant $\lambda'=0.67815\dots$ can be achieved with similar methods. The main goal of this paper is to improve Theorem \ref{thm_csy}.

\begin{theorem}\label{legit}
    Every oriented digraph has a $\gamma$-Seymour vertex, where $\gamma=0.715538\dots$ is the unique real root of the equation $8x^5+4x^4-12x^3-7x^2+2x+4=0$ in $[0, 1]$.
\end{theorem}

Our improvement comes from considering the third neighborhoods and introducing a new notion of a weighted out-degree minimizer, which we will elaborate on later. The remainder of this paper is structured as follows. In Section \ref{sec_reduction}, we show that the absence of $\gamma$-Seymour vertices implies the satisfiability of a certain quadratic constraint satisfaction problem (CSP). Section \ref{sec_opt} establishes that this CSP is indeed unsatisfiable, and thus Theorem \ref{legit} follows immediately. The final section contains some concluding remarks and further discussions.

\section{Reduction to a CSP}\label{sec_reduction}
\subsection{Small third neighborhoods}
Throughout this paper, a digraph $D$ is said to be a $\mu$-\textit{counterexample} if it has no $\mu$-Seymour vertices. We call it an \textit{edge-minimal} $\mu$-\textit{counterexample} if in addition none of its proper subgraphs is a $\mu$-counterexample. Our goal is to prove that $\gamma$-counterexamples cannot exist. We first show that in an edge-minimal $\mu$-counterexample, not only are the second neighborhoods small compared to the first neighborhoods, but the third neighborhoods are also small compared to the second neighborhoods.
\begin{lemma}\label{3ncoro}
    If $D$ is an edge-minimal $\mu$-counterexample, then for all $u\in V(D)$, $d^{+++}(u)\le\mu d^{++}(u)<\mu^2d^+(u)$.
\end{lemma}
\Cref{3ncoro} is a special case of Lemma 4 of \cite{subset} letting $S=N^+(u)$. For completeness, we include a proof with a slight adjustment to the definition of a key term $N^+(A)$ for technical clarity, while following their proof idea.
\begin{proof}
    The strict inequality is obvious, as otherwise $u$ would be a $\mu$-Seymour vertex. Thus, it suffices to show the first inequality.
    
    For a set of vertices $A\subset V(D)$, let $N^+(A)$ be their out-neighbors outside~$A$: 
    \eq{N^+(A)=\{v\in V(D)\-A:\exists u\in A\text{ s.t. }(u,v)\in D\}.} Note that $N^+(N^+(u))=N^{++}(u)$ and $N^+(N^{++}(u))\-N^+(u)=N^{+++}(u)$. Let $T$ be a maximal subset of $N^{++}(u)$ such that $\mu|T|\ge|N^+(T)\-N^+(u)|$. Note that this condition is always satisfied by the empty set. If $T=N^{++}(u)$, we have $\mu d^{++}(u) \ge d^{+++}(u)$, implying the first inequality. Otherwise, let $T'=N^{++}(u)\-T\neq\emptyset$.

    We create a new digraph $D'$ from $D$ by removing the arcs from $N^+(u)$ to $T'$. Define the notions $\distprime(v_1,v_2)$, $N'^+(v)$, $N'^{++}(v)$, $d'^+(v)$ and $d'^{++}(v)$ in the same way as in the introduction, but using $D'$ as the underlying digraph. Since in $D$, every vertex in $T'$ is reachable from $N^+(u)$, $D'$ has strictly fewer edges than~$D$, so it cannot be a $\mu$-counterexample. Thus, it admits a $\mu$-Seymour vertex $v$. Since~$v$ becomes $\mu$-Seymour  only after the edge removal, we have \eqn{\begin{aligned}N'^+(v)&\subsetneq N^+(v),\\d^{++}(v)-d'^{++}(v)&<\mu(d^+(v)-d'^+(v)).\end{aligned}\tag{*}\label{seesaw}} 
    
    The first line of \eqref{seesaw} implies that $v\in N^+(u)$. Define \eq{A&=\{y\in N^+(v): \distprime(v,y)\ge3\}\subset T',\\B&=\{y\in N^+(v): \distprime(v,y)=2\}\subset T',\\C&=N^+(A\cup B)\-(N^+(u)\cup T\cup N^+(T)).} Note that $A\cup B=N^+(v)\-N'^+(v)$, which is nonempty by the first line of \eqref{seesaw}. Also, $C\subset N^{++}(v)\-N'^{++}(v)$ and $B=N'^{++}(v)\-N^{++}(v)$. Thus, by the second line of \eqref{seesaw}, \eq{|C|-|B|\le |N^{++}(v)|-|N'^{++}(v)|<\mu(|N^+(v)|-|N'^+(v)|)=\mu(|A|+|B|).}

    Set $T_2=T\cup A\cup B$. We claim that \eq{(N^+(T_2)\-N^+(u))\cup B\subset (N^+(T)\-N^+(u))\cup C.} For all $y\in N^+(T_2)\-N^+(u)$, because it is in $N^+(T_2)$, it is either in $N^+(T)$ or in $N^+(A\cup B)\- (N^+(T)\cup T_2)$. In the first case, $y\in N^+(T)\-N^+(u)$; in the second case, $y\in C$. For all $y\in B$, note that $y\in N'^{++}(v)$ but $y\in T'$, implying that $y\in N^+(T)$ because in $D'$, $v$ cannot reach $y$ in two steps without going through~$T$. Since $T'$ is disjoint from $N^+(u)$, we have that $y\in N^+(T)\-N^+(u)$.

    Since $B\subset T_2$ is disjoint from $N^+(T_2)\-N^+(u)$, we have \eq{|N^+(T_2)\-N^+(u)|&\le |N^+(T)\-N^+(u)|+|C|-|B|\\
    &< \mu|T|+\mu(|A|+|B|)\\
    &=\mu|T_2|.} 
    Since $A\cup B\neq\emptyset$, this contradicts with the maximality of $T$. Hence, it must be that the maximal $T$ is $N^{++}(u)$, from which we get the desired inequalities.
\end{proof}
\subsection{CSP from a $\mu$-counterexample}
We show that for all $\mu$, the existence of a $\mu$-counterexample implies the satisfiability of a specific CSP. In the next section, we will see that this CSP is unsatisfiable when $\mu$ is set to the constant $\gamma$ from Theorem \ref{legit}. The proof presented in  \cite{csy} can also be interpreted in this way. The CSP we discuss below is stronger, so the phase transition occurs at a larger $\mu$. 
The key factor contributing to this improvement is that after choosing the first vertex $u$ to be the out-degree minimizer, instead of invoking the induction hypothesis in $N^+(u)$ to obtain the second vertex $v$, we choose our second vertex to be a ``weighted out-degree minimizer'' in $N^+(u)$. That is, we choose in $N^+(u)$ a vertex that minimizes the weighted out-degree, giving a higher weight to the out-neighbors in $N^+(u)$. This modification, together with \Cref{3ncoro}, provides the desired strengthening of the CSP. The value of the weighting factor $w$ will not be chosen until \Cref{sec_opt}, so our next result works for all $w \ge 1$.

\begin{lemma}\label{feas}
    If for some $\mu\ge0$ there exists a $\mu$-counterexample, then for all constants $w\ge1$, the following CSP is satisfiable:
    \eqn{x_{11},x_{12},x_{13},x_{14},x_{21},x_{22},x_{23},x_{24},x_{32},x_{33},x_{34}\in \R\nonumber\\\text{subject to}\nonumber\\\mu(x_{11}+x_{12}+x_{13}+x_{14})> x_{21}+x_{22}+x_{23}+x_{24},\label{i1}\\
    \mu^2(x_{11}+x_{12}+x_{13}+x_{14}) > x_{32}+x_{33}+x_{34},\label{i2}\\
    \mu(x_{11}+x_{21}) > x_{12}+x_{22}+x_{32},\label{i3}\\
    x_{21}\ge x_{12}+x_{13}+x_{14},\label{i4}\\
    x_{14},x_{24},x_{34},x_{23}\ge 0\label{xi4},\\x_{11}>0,\quad x_{13},x_{22},x_{32},x_{33},x_{12}+x_{13},x_{12}+x_{22}\ge 0,\label{importantlinear}\\
    F> 0,\label{i5}
    }where 
    \eq{\begin{aligned}F&=x_{21}(x_{32}-x_{11}-x_{13})+x_{22}(x_{32}+x_{23}+x_{33})+x_{23}x_{12}-x_{14}(x_{12}+x_{21}+x_{22})\\
        &~~~+\frac{(x_{21}+x_{22})^2}2-\frac{wx_{11}^2+x_{12}^2}2+(w-1)x_{11}x_{12}. \end{aligned}}
\end{lemma}
\begin{proof}
    Take an edge-minimal $\mu$-counterexample,~$D$. Using the constant~$w$ as a weighting factor, we will select two vertices from $D$ in the following way: first let $u\in V(D)$ be a vertex with minimum out-degree: i.e., $d^+(y)\ge d^+(u)$ for all $y\in V(D)$. Then, let $X_1=N^+(u)$ and $X_2=N^{++}(u)$. Note that $X_1\neq\emptyset$ as otherwise $d^+(u)=0$, making $u$ a $\mu$-Seymour vertex. Among those vertices in $X_1$, let $v$ be a vertex that minimizes the quantity $wd^+_{X_1}(v)+d^+_{X_2}(v)$, where $d_A^+(v)=|N^+(v)\cap A|$ denotes the number of out-neighbors of $v$ in $A$. We remark that $d^+_{X_1}(v)+d^+_{X_2}(v)=d^+(v)$ because $N^+(v)\subset X_1\cup X_2$, and that $v$ is the weighted out-degree minimizer we mentioned at the beginning of this subsection. 
    For $i,j\in\{1,2,3\}$, define the following pair-wise disjoint subsets of vertices based on their positive distances from $u$ and $v$:    
    \eq{X_{ij}=\{y\in V(D):\dist(u,y)=i\text{ and }\dist(v,y)=j\}.} 
    For $i\in\{1,2,3\}$, we also set 
    $$X_{i4}={\{y\in V(D):\dist(u,y)=i\text{ and }\dist(v,y)\ge4\}}.$$ 
    Hence, the sets $\{X_{ij}\}_{i\in\{1,2,3\},j\in\{1,2,3,4\}}$ partition $N^+(u)\cup N^{++}(u)\cup N^{+++}(u)$. Note that since $(u,v)\in D$, $X_{31}=\emptyset$. For every other~$X_{ij}$ we defined, let $x_{ij}=|X_{ij}|\in\R$. We claim that this assignment solves the CSP. 

    Note that \eq{d^+(u)&=|X_1|=x_{11}+x_{12}+x_{13}+x_{14},\\d^{++}(u)&=|X_2|=x_{21}+x_{22}+x_{23}+x_{24},\\d^{+++}(u)&=x_{32}+x_{33}+x_{34},\\d^+(v)&=x_{11}+x_{21},\\d^{++}(v)&=x_{12}+x_{22}+x_{32}.}
    Thus, the constraints \eqref{i1} and \eqref{i2} follow from \Cref{3ncoro}, and \eqref{i3} follows from the fact that $v$ is not a $\mu$-Seymour vertex. Since $u$ minimizes the out-degree, we have $d^+(v)\ge d^+(u)$, implying the constraint \eqref{i4}. 
    Observe that $x_{11}$ is strictly positive because otherwise, $N^+(v)\subset N^{++}(u)$ and $u$ would be a $1$-Seymour vertex: $d^{++}(u)\ge d^+(v)\ge d^+(u)$. The rest of the constraints \eqref{xi4} and \eqref{importantlinear} follow from the nonnegativity of the $x_{ij}$ variables. 
    Thus, it remains to show that the constraint \eqref{i5} holds.

    For $A,B\subset V(D)$, define $e(A,B)=|\{(a,b)\in D:a\in A,\ b\in B\}|$. Let $Y=X_{11}\cup X_{12}\cup X_{21}\cup X_{22}.$ We will count, in two ways, the quantity $e(Y, V(D))$. 
    
        We first give a lower bound on $e(Y,V(D))$.
        Note that for all $y\in V(D)$, $d^+(y)\ge d^+(u)$. Applying this bound to each vertex in $X_{12}\cup X_{21}\cup X_{22}$, we have \eq{e(X_{12}\cup X_{21}\cup X_{22}, V(D))\ge (x_{12}+x_{21}+x_{22})|X_1|.} For each vertex $y\in X_{11}$, we instead use the fact that $y$ failed to beat~$v$ when we chose the weighted out-degree minimizer inside $X_1$. That is, \eq{wd^+_{X_1}(y)+d^+_{X_2}(y)\ge wd^+_{X_1}(v)+d^+_{X_2}(v)=wx_{11}+x_{21}.} It follows that $d^+(y)\ge wx_{11}+x_{21}-(w-1)d^+_{X_1}(y).$ Summing over all $y\in X_{11}$, we have \eq{e(X_{11},V(D))&\ge (wx_{11}+x_{21})x_{11}-(w-1)e(X_{11},X_1)\\&=(wx_{11}+x_{21})x_{11}-(w-1)e(X_{11},X_{11}\cup X_{12})\\&\ge (wx_{11}+x_{21})x_{11}-(w-1)(x_{11}^2/2+x_{11}x_{12}),}
        where the last line follows from the fact that $e(X_{11},X_{11})\le{|X_{11}|\choose2}\le x_{11}^2/2$ and the assumption $w \ge 1$. 
        Hence, \eq{e(Y,V(D))\ge (x_{12}+x_{21}+x_{22})|X_1|+(wx_{11}+x_{21})x_{11}-(w-1)(x_{11}^2/2+x_{11}x_{12}).}
        
        Next, we give an upper bound on $e(Y,V(D))$.
        Note that $Y\supset X_{11}\neq\emptyset$. Using $\binom{a}{2} < a^2/2$ for all $a>0$, we have \eq{e(Y,Y)\le{|Y|\choose2}<\frac{|Y|^2}{2}=\frac{(x_{11}+x_{12}+x_{21}+x_{22})^2}2.} The number of arcs from $Y$ to $V(D)\-Y$ can be estimated as follows.
        \begin{itemize}
            \item $e(X_{11}, V(D)\-Y)=0$, since all out-neighbors of vertices in $X_{11}$ lie in~$Y$. 
            \item $e(X_{21}, V(D)\-Y)\le x_{21}x_{32}$, since all out-neighbors of vertices in $X_{21}$ lie in $Y\cup X_{32}$. 
            \item $e(X_{12}, V(D)\-Y)\le x_{12}(x_{13}+x_{23})$, since all out-neighbors of vertices in~$X_{12}$ lie in $Y\cup X_{13}\cup X_{23}$. 
            \item $e(X_{22}, V(D)\-Y)\le x_{22}(x_{32}+x_{13}+x_{23}+x_{33})$, since all out-neighbors of vertices in $X_{22}$ lie in $Y\cup X_{32}\cup X_{13}\cup X_{23}\cup X_{33}$. 
        \end{itemize}
        Hence, \eq{e(Y,V(D))< &\ \frac{(x_{11}+x_{12}+x_{21}+x_{22})^2}2+x_{21}x_{32}+x_{12}(x_{13}+x_{23})\\&+x_{22}(x_{32}+x_{13}+x_{23}+x_{33}).}
        Together with the lower bound and rearranging some terms, we have $F> 0$.
\end{proof}
\subsection{Adjustments}
Now we claim that, given an assignment of the variables that solves the CSP in \Cref{feas}, we can tweak their values so that they satisfy a slightly different CSP. This will help us reduce the number of independent variables. We will need a narrower range for $w$, but it will still contain the value we eventually choose in \Cref{sec_opt}.
\begin{lemma}\label{adjust}
    For all $\mu\ge0$ and $w\in(1,1+\mu^2)$, if the CSP in \Cref{feas} is satisfiable, then the following CSP is also satisfiable. Here, the quantity $F$ is defined in the same way as in \Cref{feas}.
    \eqn{x_{11},x_{12},x_{13},x_{14},x_{21},x_{22},x_{23},x_{24},x_{32},x_{33},x_{34}\in \R\nonumber\\\text{subject to}\nonumber\\\mu(x_{11}+x_{12}+x_{13}+x_{14})= x_{21}+x_{22}+x_{23}+x_{24},\tag{1=}\label{i1e}\\
    \mu^2(x_{11}+x_{12}+x_{13}+x_{14}) = x_{32}+x_{33}+x_{34},\tag{2=}\label{i2e}\\
    \mu(x_{11}+x_{21}) = x_{12}+x_{22}+x_{32},\tag{3=}\label{i3e}\\
    x_{21}\ge x_{12}+x_{13}+x_{14},\tag{4}\\
    x_{14},x_{24},x_{34},x_{23}= 0,\tag{5=}\label{xi4e}\\x_{11}>0,\quad x_{13},x_{22},x_{32},x_{33},x_{12}+x_{13},x_{12}+x_{22}\ge 0,\tag{6}\\
    F> 0.\tag{7}\label{i5s}
    }
\end{lemma}
\begin{proof}
    Take an assignment of the variables $\{x_{ij}\}$ that satisfies the CSP in \Cref{feas}. We will perform a multi-step procedure to turn it into a solution of the desired CSP.
    \begin{enumerate}
        \item Increase the variable $x_{13}$ by $x_{14}$, and set $x_{14}$, $x_{24}$ and $x_{34}$ to 0. Note that doing so will preserve $x_{13}+x_{14}$ and increase $F$ by a nonnegative amount $(x_{12}+x_{22})x_{14}$. It follows that after this step, the variables will still satisfy the constraints \eqref{i1}--\eqref{i5}, while we can additionally assume that $x_{14}=x_{24}=x_{34}=0$. 
        \item Increase the values of $x_{23}$ and $x_{33}$, until the constraints \eqref{i1e} and \eqref{i2e} are satisfied. Since $F$ is an increasing function w.r.t $x_{23}$ and $x_{33}$, doing so will preserve the constraints \eqref{i3}--\eqref{i5}. Thus, after this step the variables will satisfy the constraints \eqref{i1e}, \eqref{i2e}, \eqref{i3}--\eqref{i5}, as well as $x_{14}=x_{24}=x_{34}=0$.
        \item Greedily increase the values of $x_{13}$ and $x_{22}$, and decrease the values of $x_{12}$ and $x_{23}$, all at the same rate. That is, we increase $x_{13}$ and $x_{22}$ and decrease $x_{12}$ and $x_{23}$ by some small value $\delta$ all at once, and continue doing so as much as the constraints allow. Note that doing this will preserve the constraints \eqref{i1e}, \eqref{i2e}, \eqref{i3}, \eqref{i4} and \eqref{importantlinear}. The amount that the quantity~$F$ changes by is a function of $\delta$. Because $F$ is a quadratic form in the variables, this function has the form $a\delta+b\delta^2$. Since the step size $\delta$ can be made arbitrarily small, to see that $F$ increases it suffices to show that~$a$, the total differential along the direction of change, stays strictly positive: \eq{a&=\frac{\partial F}{\partial x_{13}}+\frac{\partial F}{\partial x_{22}}-\frac{\partial F}{\partial x_{12}}-\frac{\partial F}{\partial x_{23}}\\&=-(w-1)x_{11}+x_{32}+x_{33}\\&>-\mu^2x_{11}+x_{32}+x_{33}\tag{$x_{11}>0,\ w-1<\mu^2$}\\&\ge-\mu^2(x_{11}+x_{12}+x_{13})+x_{32}+x_{33}\tag{$x_{12}+x_{13}\ge0$}\\&=0.\tag{\eqref{i2e} and $x_{14}=x_{34}=0$}} 
        Thus, the constraint \eqref{i5} is also preserved. With only the constraint \eqref{xi4} holding against us, we will stop precisely when $x_{23}=0$. Hence, after this step, the variables will satisfy the constraints \eqref{i1e}, \eqref{i2e}, \eqref{i3}, \eqref{i4}, \eqref{xi4e}, \eqref{importantlinear} and \eqref{i5}.
        \item The goal of this step is to satisfy the constraint \eqref{i3e} while preserving \eqref{i1e}, \eqref{i2e}, \eqref{i4}, \eqref{xi4e}, \eqref{importantlinear} and \eqref{i5}. 
        \begin{enumerate}
            \item Greedily decrease $x_{21}$ and increase $x_{22}$ at the same rate. Note that doing this will preserve the constraints \eqref{i1e}, \eqref{i2e}, \eqref{xi4e} and \eqref{importantlinear}, and only increase the value of $F$: \eq{\frac{\partial F}{\partial x_{22}}-\frac{\partial F}{\partial x_{21}}=x_{11}+x_{13}+x_{23}+x_{33}\ge x_{11}>0.} Thus, either we will satisfy \eqref{i3e} before violating \eqref{i4}, or we are forced to stop because the constraint \eqref{i4} is at equality: \eqn{x_{21}-x_{12}-x_{13}-x_{14}=0.\tag{4=}\label{i4e}} In the former case we terminate; otherwise we proceed to the next substep.
            \item Greedily increase $x_{12}$ and decrease $x_{13}$ at the same rate. Doing so will preserve the constraints \eqref{i1e}, \eqref{i2e}, \eqref{i4e} and \eqref{xi4e}, and only increase the value of $F$: \eq{\frac{\partial F}{\partial x_{12}}-\frac{\partial F}{\partial x_{13}}=(w-1)x_{11}+x_{21}-x_{12}-x_{14}+x_{23}\ge(w-1)x_{11}>0.} The first inequality above follows from \eqref{i4e}, \eqref{xi4e} and $x_{13}\ge0$. Thus, either we will satisfy \eqref{i3e} before violating \eqref{importantlinear}, or we are forced to stop because $x_{13}=0$. 
            Actually the latter would not occur earlier, since 
            when $x_{13}=0$, the strict inequality in constraint~\eqref{i3} is violated:
            \eq{\mu(x_{11}+x_{21})&=\mu(x_{11}+x_{12}+x_{13}+x_{14})\tag{by \eqref{i4e}}\\&=x_{21}+x_{22}+x_{23}+x_{24}\tag{by \eqref{i1e}}\\&=x_{12}+x_{13}+x_{22}\tag{by \eqref{i4e} and \eqref{xi4e}}\\&\le x_{12}+x_{22}+x_{32}.\tag{$x_{13}=0\le x_{22}$}}
            Hence, we will always satisfy \eqref{i3e} before violating \eqref{importantlinear}.
        \end{enumerate}
    \end{enumerate}
    Thus, after the adjustments, the variables will satisfy the constraints \eqref{i1e}, \eqref{i2e}, \eqref{i3e}, \eqref{i4}, \eqref{xi4e}, \eqref{importantlinear} and \eqref{i5}.
    Hence, the desired CSP is satisfiable.
\end{proof}
\section{Proof of Theorem \ref{legit}}\label{sec_opt}
We are ready to prove \Cref{legit}. Let $\gamma=0.715538\dots$ be the unique real root of $8x^5+4x^4-12x^3-7x^2+2x+4=0$ in the interval $[0, 1]$. By \Cref{feas} and \Cref{adjust}, \Cref{legit} immediately follows from the next result.
\begin{lemma}\label{infeas}
    Letting $\mu=\gamma$ and $w=\gamma^2+2\gamma^3\in(1,1+\gamma^2)$, the CSP in \Cref{adjust} is unsatisfiable.
\end{lemma}
\begin{proof}
    Note that $\gamma^2+2\gamma^3\approx1.2447$ and $\gamma^2\approx0.5119$, so we have $w\in(1,1+\gamma^2)$.
    
    Take an assignment of the variables $\{x_{ij}\}$ that satisfies the CSP in \Cref{adjust}.
    We set $\theta=2+2\gamma-4\gamma^3\approx1.9657$, and define four more variables:
    \eqn{y_1=\frac{x_{32}}{\gamma\theta}-\frac{x_{12}}\gamma,\ \ 
    y_2=\frac{x_{22}}\gamma-\frac{x_{33}}{\gamma^2},\ \ 
    y_3=\frac{x_{32}}{\gamma\theta},\ \ 
    y_4=\frac{x_{33}}{\gamma^2}.\label{ydef}}
    By the constraints \eqref{i1e}, \eqref{i2e}, \eqref{i3e} and \eqref{xi4e} and the definitions in \eqref{ydef}, all the $x_{ij}$ variables are linear combinations of $y_1$, $y_2$, $y_3$ and $y_4$:
    \eq{x_{32}&=\gamma\theta y_3,\ x_{33}=\gamma^2y_4,\ x_{12}=\gamma(y_3-y_1),\ x_{22}=\gamma(y_2+y_4),\tag{by \eqref{ydef}}\\
    x_{14}&=x_{24}=x_{34}=x_{23}=0,\tag{by \eqref{xi4e}}\\
    x_{21}&=\frac{x_{32}+x_{33}+x_{34}}\gamma-x_{22}-x_{23}-x_{24}=\theta y_3-\gamma y_2,\tag{by \eqref{i1e} and \eqref{i2e}}\\
    x_{11}&=(x_{11}+x_{21})+x_{22}+x_{23}+x_{24}-\frac{x_{32}+x_{33}+x_{34}}\gamma\tag{by \eqref{i1e} and \eqref{i2e}}\\
    &=\frac{x_{12}+x_{22}-x_{33}}\gamma+x_{22}+x_{23}\tag{by \eqref{i3e} and \eqref{xi4e}}\\
    &=(1+\gamma)y_2+y_3+y_4-y_1,\\
    x_{13}&=\frac{x_{21}+x_{22}+x_{23}+x_{24}}\gamma-x_{11}-x_{12}-x_{14}\tag{by \eqref{i1e}}\\
    &=(1+\gamma)(y_1-y_2)+(\gamma^{-1}\theta-\gamma-1) y_3.}
    
    By substitution, we have
    \eq{F&=\sum_{1\le i\le j\le 4}c_{ij}y_iy_j,
    }
    where the coefficients are
    {\allowdisplaybreaks[1] 
    \eq{
        c_{11}&=(w-1)\gamma-(w+\gamma^2)/2,\\*
        c_{12}&=\gamma^2+(1+\gamma)(w-(w-1)\gamma),\\*
        c_{13}&=w+\gamma^2-\gamma(2w-2+\theta),\\*
        c_{14}&=w-(w-1)\gamma,\\
        c_{22}&=-w(1+\gamma)^2/2,\\*
        c_{23}&=\theta-\gamma^2-(1+\gamma)(w-(w-1)\gamma),\\*
        c_{24}&=\gamma+\gamma^3-w(1+\gamma),\\
        c_{33}&=\gamma(w-1+\theta)+(\gamma-\gamma^{-1})\theta^2+(\theta^2-\gamma^2-w)/2,\\*
        c_{34}&=(\gamma+\gamma^2-1)\theta-w+(w-1)\gamma,\\
        c_{44}&=\gamma^3+(\gamma^2-w)/2.
    }}
    
    We note that $c_{13}=c_{33}=c_{44}=0$. The fact that $c_{13}=c_{44}=0$ follows from the definitions of $w$ and $\theta$. Substituting in those definitions, we also have \eq{c_{33}=\gamma^{-1}(\gamma-1)(2\gamma^2+2\gamma+1)(8\gamma^5+4\gamma^4-12\gamma^3-7\gamma^2+2\gamma+4)=0,} by the choice of $\gamma$. Now, set $\rho=1+\theta-\gamma^{-1}\theta \approx0.2186$  and consider four more quantities, which are nonnegative by the constraints \eqref{i4} and \eqref{importantlinear}: \eqon{P_1&=(y_2-y_1+\rho y_3)(y_2+y_4)=\gamma^{-1}(x_{21}-x_{12}-x_{13}-x_{14})x_{22}\ge0,\\
    P_2&=(\gamma^{-1}\theta y_3-y_2)y_4=\gamma^{-3}x_{21}x_{33}\ge0,\\
    P_3&=(y_2+y_4)y_3=\gamma^{-2}\theta^{-1}x_{22}x_{32}\ge0,\\
    P_4&=y_3y_4=\gamma^{-3}\theta^{-1}x_{32}x_{33}\ge0.\label{ppos}}
    An easier way to verify \eqref{ppos} is to start from the RHS and substitute the $x_{ij}$'s with linear combinations of $y_1$, $y_2$, $y_3$ and~$y_4$ shown earlier. Observe that the $P_i$'s and~$F$ have now become quadratic forms in the variables $y_1$, $y_2$, $y_3$ and~$y_4$, all with coefficient zero on $y_1y_3$, $y_3^2$ and~$y_4^2$. It turns out that one can linearly combine them to obtain a quadratic form in just the variables~$y_1$ and~$y_2$, which we demonstrate below. In fact, such linear combination is unique up to a constant factor, although our proof will not rely on this. {
    \eqn{&F+c_{14}P_1+(c_{14}+c_{24})P_2-(c_{23}+\rho c_{14})P_3+(c_{23}-c_{34}-\gamma^{-1}\theta(c_{14}+c_{24}))P_4\nonumber\\
    ={}&c_{11}y_1^2+(c_{12}-c_{14})y_1y_2+(c_{22}+c_{14})y_2^2.\label{false}}}
    
    We claim that this leads to a contradiction. Note that in \eqref{false}, all the $P_i$'s have positive coefficients: \eq{&c_{14}\approx 1.0696,\quad c_{14}+c_{24}\approx 0.0162,\quad {-(c_{23}+\rho c_{14})}\approx 0.1475,\\&c_{23}-c_{34}-\gamma^{-1}\theta(c_{14}+c_{24})\approx 0.1967,} so it follows from~\eqref{i5s} and~\eqref{ppos} that the LHS of~\eqref{false} must be strictly positive. However, the RHS of \eqref{false} is nonpositive, due to a negative leading coefficient and a negative discriminant: \eq{c_{11}\approx -0.7033,\quad (c_{12}-c_{14})^2-4c_{11}(c_{22}+c_{14})\approx -0.5120.}
    This is a contradiction. Hence, the CSP in \Cref{adjust} is unsatisfiable.
\end{proof}

\section{Concluding Remarks}
One can show that the constant $\gamma=0.715538\dots$ in \Cref{legit} is optimal under the constraints employed in our analysis. However, better bounds are observed when more nuanced constraints are added. Note that in the proof of \Cref{feas}, we selected a global out-degree minimizer, $u$, and then a weighted out-degree minimizer $v$ in $X_1=N^+(u)$. Continuing this procedure, one can further select a weighted out-degree minimizer $v_2$ in $X_{11}$, and use the more refined vertex sets \eq{X_{ijk}=\{y\in V(D):\dist(u,y)=i,\ \dist(v,y)=j,\ \dist(v_2,y)=k\}.} The analogous CSP then becomes unsatisfiable even when $\mu=0.74530$, with optimally chosen weighting factors. This follows from the rigorous bounds computed by the Gurobi optimization software. However, we believe that turning it into a fully self-contained, human-checkable proof would be excessively laborious.

It is natural to ask if further improvements to \Cref{legit}, or even a proof of \Cref{main_conj} in full generality, can be obtained in a similar fashion: one chooses a sequence of vertices that altogether satisfies some extremal properties (in our case, $u$ minimizes the out-degree, and $v$ minimizes the weighted out-degree among those in $N^+(u)$), and derives a contradiction provided that all the vertices in the sequence violate the $\mu$-Seymour condition. Alternatively, one might wonder if there is a theoretical limit to this approach. The median order proof \cite{havet2000median} of Dean's conjecture can also be loosely placed under this framework.

Our main result also has connections with the \textit{Caccetta--H\"aggkvist conjecture}, which we state below. We refer interested readers to the excellent survey \cite{sullivan2006summary}.
\begin{conjecture}[\cite{caccetta1978minimal}]\label{cacc}
    Every $n$-vertex oriented digraph with minimum out-degree at least~$r$ has a directed cycle with length at most $\ceil{\frac nr}$.
\end{conjecture}

The case $r=n/2$ is trivial, but the next case $r=n/3$ is still open, as well as a weakening of that case known as the \textit{Behzad--Chartrand--Wall conjecture}: 
\begin{conjecture}[\cite{behzad1970minimal}]\label{bcw}
    Every $n$-vertex oriented digraph with minimum out-degree and minimum in-degree at least $n/3$ contains a directed triangle.
\end{conjecture}
The extremal examples for the second neighborhood conjecture we mentioned also show that \Cref{bcw}, if true, is sharp. Note that \Cref{main_conj} would imply \Cref{bcw}. In fact, by Proposition~5 of~\cite{csy}, our bound in \Cref{legit} also yields the following bound on the Behzad--Chartrand--Wall conjecture:
\begin{corollary}\label{ourcacc}
    Every $n$-vertex oriented digraph with minimum out-degree and minimum in-degree at least $n/(2+\gamma)$ contains a directed triangle, where $\gamma=0.715538\dots$ is the same as in \Cref{legit}.
\end{corollary}
The resulting constant factor $1/(2+\gamma)$ is approximately $0.3683$, which is not as good as the current record $0.3465$ shown by Hladk\`y et al. \cite{hladky2017counting}. However, it might be of independent interest to apply the ideas and tools we developed in this paper to study \Cref{cacc} or \Cref{bcw} directly.
\bibliographystyle{plain}
\bibliography{bib.bib}

\begin{thebibliography}{10}

\bibitem{ai2024seymour}
Jiangdong Ai, Stefanie Gerke, Gregory Gutin, Shujing Wang, Anders Yeo, and Yacong Zhou.
\newblock On {Seymour}'s and {Sullivan}'s second neighbourhood conjectures.
\newblock {\em Journal of Graph Theory}, 105(3):413--426, 2024.

\bibitem{behzad1970minimal}
Mehdi Behzad, Gary Chartrand, and Curtiss Wall.
\newblock On minimal regular digraphs with given girth.
\newblock {\em Fundamenta Mathematicae}, 69:227--231, 1970.

\bibitem{brantner2009contributions}
James Brantner, Greg Brockman, Bill Kay, and Emma Snively.
\newblock Contributions to {Seymour}’s second neighborhood conjecture.
\newblock {\em Involve, a Journal of Mathematics}, 2(4):387--395, 2009.

\bibitem{caccetta1978minimal}
Louis Caccetta and Roland Haggkvist.
\newblock {\em On minimal digraphs with given girth}.
\newblock Department of Combinatorics and Optimization, University of Waterloo, 1978.

\bibitem{csy}
Guantao Chen, Jian Shen, and Raphael Yuster.
\newblock Second neighborhood via first neighborhood in digraphs.
\newblock {\em Ann Comb}, 7:15--20, 06 2003.

\bibitem{cohn2016number}
Zachary Cohn, Anant Godbole, Elizabeth~Wright Harkness, and Yiguang Zhang.
\newblock The number of seymour vertices in random tournaments and digraphs.
\newblock {\em Graphs and Combinatorics}, 32:1805--1816, 2016.

\bibitem{dara2022extending}
Suresh Dara, Mathew~C Francis, Dalu Jacob, and N~Narayanan.
\newblock Extending some results on the second neighborhood conjecture.
\newblock {\em Discrete Applied Mathematics}, 311:1--17, 2022.

\bibitem{dean1995squaring}
Nathaniel Dean and Brenda~J Latka.
\newblock Squaring the tournament---an open problem.
\newblock {\em Congressus Numerantium}, pages 73--80, 1995.

\bibitem{espuny2024seymour}
Alberto Espuny~D{\'\i}az, Ant{\'o}nio Gir{\~a}o, Bertille Granet, and Gal Kronenberg.
\newblock Seymour's second neighbourhood conjecture: random graphs and reductions.
\newblock {\em Random Structures \& Algorithms}, 2024.

\bibitem{fidler2007remarks}
Dror Fidler and Raphael Yuster.
\newblock Remarks on the second neighborhood problem.
\newblock {\em Journal of Graph Theory}, 55(3):208--220, 2007.

\bibitem{fisher1996squaring}
David~C Fisher.
\newblock Squaring a tournament: a proof of {Dean}'s conjecture.
\newblock {\em Journal of Graph Theory}, 23(1):43--48, 1996.

\bibitem{ghazal2012seymour}
Salman Ghazal.
\newblock Seymour's second neighborhood conjecture for tournaments missing a generalized star.
\newblock {\em Journal of Graph Theory}, 71(1):89--94, 2012.

\bibitem{havet2000median}
Fr{\'e}d{\'e}ric Havet and St{\'e}phan Thomass{\'e}.
\newblock Median orders of tournaments: a tool for the second neighborhood problem and {Sumner}'s conjecture.
\newblock {\em Journal of Graph Theory}, 35(4):244--256, 2000.

\bibitem{hernandez2012k}
C{\'e}sar Hern{\'a}ndez-Cruz and Hortensia Galeana-S{\'a}nchez.
\newblock $k$-kernels in $k$-transitive and $k$-quasi-transitive digraphs.
\newblock {\em Discrete Mathematics}, 312(16):2522--2530, 2012.

\bibitem{hladky2017counting}
Jan Hladk{\`y}, Daniel Kr{\'a}l’, and Sergey Norin.
\newblock Counting flags in triangle-free digraphs.
\newblock {\em Combinatorica}, 37(1):49--76, 2017.

\bibitem{kaneko2001minimum}
Yoshihiro Kaneko and Stephen~C Locke.
\newblock The minimum degree approach for {Paul} {Seymour}'s distance 2 conjecture.
\newblock {\em Congressus Numerantium}, pages 201--206, 2001.

\bibitem{liang2017seymour}
Hao Liang and Jun-Ming Xu.
\newblock On {Seymour}’s second neighborhood conjecture of m-free digraphs.
\newblock {\em Discrete Mathematics}, 340(8):1944--1949, 2017.

\bibitem{lim2020generalisation}
Jeck Lim.
\newblock A generalisation of {Seymour}'s second neighbourhood conjecture.
\newblock {\em arXiv preprint \tt arxiv:2001.07242}, 2020.

\bibitem{llado2013second}
Anna Llad{\'o}.
\newblock On the second neighborhood conjecture of {Seymour} for regular digraphs with almost optimal connectivity.
\newblock {\em European Journal of Combinatorics}, 34(8):1406--1410, 2013.

\bibitem{mezher2024note}
Dania Mezher and Moussa Daamouch.
\newblock A note on the second neighborhood problem for $ k $-anti-transitive and $ m $-free digraphs.
\newblock {\em arXiv preprint \tt arxiv:2405.17797}, 2024.

\bibitem{subset}
Tyler Seacrest.
\newblock Seymour's second neighborhood conjecture for subsets of vertices.
\newblock {\em arXiv preprint \tt arXiv:1808.06293v3}, 2018.

\bibitem{sullivan2006summary}
Blair~Dowling Sullivan.
\newblock A summary of results and problems related to the {Caccetta}-{H{\"a}ggkvist} conjecture.
\newblock {\em arXiv preprint \tt arXiv:math/0605646}, 2006.

\end{thebibliography}
\end{document}